\documentclass[reqno]{amsart}
\usepackage[all,ps,cmtip]{xy}

\usepackage{amssymb}
\usepackage{mathrsfs}
\usepackage{enumerate}

\usepackage{amsmath}
\usepackage{amsthm}
\usepackage{mathrsfs}
\usepackage{ esint }
\usepackage{graphicx}
\usepackage{color}
\usepackage{bm}
\usepackage{tikz-cd}
\usetikzlibrary{cd}
\usepackage{comment}

\usepackage{stmaryrd}

\begin{document}
	\author{Alexander Lai De Oliveira}

	\title{Lazardian Witt vectors}
	\newtheorem{theoremx}{Theorem}
	\newtheorem{mainthm}{Theorem}
	\renewcommand{\themainthm}{\Alph{mainthm}}
	\newtheorem{theorem}[subsection]{Theorem}
	\newtheorem*{theorem*}{Theorem}
	\newtheorem*{corollary*}{Corollary}
	\newtheorem{corollary}[subsection]{Corollary}
	\newtheorem{lemma}[subsection]{Lemma}
	\newtheorem{proposition}[subsection]{Proposition}
	
	\theoremstyle{remark}
	\newtheorem{remark}[subsection]{Remark}
	\theoremstyle{definition}
	\newtheorem{definition}[subsection]{Definition}
	\newtheorem{example}[subsection]{Example}

	\setcounter{tocdepth}{1}
	
	\begin{abstract}
		We use Lazard's universal $\pi$-ring to construct a variation of the $\pi$-typical ramified Witt vector functors, which we call the Lazardian Witt vector functor. We then use the Lazardian Witt vector functor to construct the universal residual perfection of a Lazardian algebra.
	\end{abstract}
	\maketitle
	\tableofcontents

	\frenchspacing
	
	\section{Introduction}
	
	The original motivation for this paper was to give a choice-free construction of the universal residual perfection $A^u$ of a complete discrete valuation ring $(A,\mathfrak{m})$ with residue field $k:= A/\mathfrak{m}$ of characteristic $p > 0$, in the nontrivial case of imperfect $k$. Borger~\cite{Bor04} constructs $A^u$ as a certain infinitesimal thickening of the perfection of the polynomial ring
	\begin{equation*}
		k_T:= k[u_{t,n}\mid t \in T,n\geq 1]
	\end{equation*} 
	defined with respect to a choice of lifted $p$-basis $T \subset A$. In other words $k_T^{\operatorname{pf}}$ is a construction of the residue ring $k^u:= A^u/\mathfrak{m}A^u$ of $A^u$. 
	
	Given $k^u$, the infinitesimal thickening used to construct $A^u$ depends on a choice of uniformizer for $A$. For example, in the case of 
	\begin{equation*}
		A_{\text{equal}}:= \mathbf{F}_p(t)\llbracket\pi\rrbracket, \quad A_{\text{mixed}}:=\varprojlim_n\mathbf{Z}[t]_{(p)}/p^{n + 1}\mathbf{Z}[t]_{(p)}
	\end{equation*}
	with $k_T := \mathbf{F}_p(t)[t_1,t_2,\dots]$, one has
	\begin{equation}\label{Example:simpleURP}
		A_{\text{equal}}^u = \mathbf{F}_p(t)[t_1,t_2,\dots]^{\operatorname{pf}}\llbracket\pi\rrbracket,\quad  A_{\text{mixed}}^u = W(\mathbf{F}_p(t)[t_1,t_2,\dots]^{\operatorname{pf}}),
	\end{equation}
	where $W$ is the $p$-typical Witt vector functor. 
	
	The rings $A^u$ in (\ref{Example:simpleURP}) are explicit constructions of an initial object of a certain category introduced by Borger. We give a slightly more general definition.
	\begin{definition}\label{Def:CRP}
		Let $(A,I)$ be a pair consisting of a ring $A$ and a principal ideal $I \subset A$ containing $p$ such that $A$ is $I$-adically complete and $\operatorname{Ann}(I) = I^m$ for some $m \in \mathbf{N}\cup \{\infty\}$, where $I^\infty := 0$. Define the category $\mathsf{CRP}_A$ of \emph{complete residual perfections} of $A$ to be the full subcategory of $A$-algebras $B$ such that:
		\begin{enumerate}[\normalfont (i)]
			\item $B$ is $I B$-adically complete;
			\item $B/I B$ is a perfect $\mathbf{F}_p$-algebra;
			\item The multiplication map $I \otimes_A B \to B$ is injective.
		\end{enumerate}
		An $I$-adically complete $A$-algebra satisfying (iii) is called \emph{strict}. 
	\end{definition}
	\begin{example}\label{Ex:BabyTitling}
		The category $\mathsf{CRP}_{\mathbf{Z}_p}$ is the category of strict $p$-rings (cf.~Serre~\cite[p.~37]{Ser79}) with initial object $\mathbf{Z}_p$. Setting $U:= -\otimes_{\mathbf{Z}_p}\mathbf{F}_p$, the residue functor $U: \mathsf{CRP}_{\mathbf{Z}_p}\to \mathsf{PerfAlg}_{\mathbf{F}_p}$ is an equivalence of categories. The $p$-typical Witt vector functor $W: \mathsf{Alg}_\mathbf{Z}\to \mathsf{Alg}_\mathbf{Z}$ is a left adjoint. More precisely, we have a diagram
		\begin{equation*}
			\begin{tikzcd}
				\mathsf{Alg}_{\mathbf{Z}}\arrow[from = d, "W"]&\mathsf{Alg}_{\mathbf{Z}_p} \arrow[l]\arrow[from = d, "W"]&\mathsf{CRP}_{\mathbf{Z}_p}\arrow[l,hook']\arrow[d, shift left = 1ex, "U"] \\
				 \mathsf{Alg}_{\mathbf{Z}}& \mathsf{Alg}_{\mathbf{F}_p}\arrow[l]&\mathsf{PerfAlg}_{\mathbf{F}_p} \arrow[l,hook']\arrow[u, "W", shift left = 1ex, "\dashv"']
			\end{tikzcd}
		\end{equation*}
		in which the two functors in the right column form an adjoint equivalence.
	\end{example}
	Since $A^u$ is an initial object of $\mathsf{CRP}_A$, the explicit constructions in (\ref{Example:simpleURP}) are independent of both the choice of lifted $p$-basis and uniformizer up to unique $A$-algebra isomorphism. The goal was to give a construction that makes the independence of the choice of lifted $p$-basis manifestly clear by avoiding choices altogether: choice-free as opposed to choice-independent. We were only successful in removing an explicit choice of $p$-basis; our construction of functors relies on a choice of uniformizer. 
	
	While the existence of $A^u$ can be guaranteed by the general adjoint functor theorem in settings slightly more general than with respect to pairs $(A,I)$ as in Definition~\ref{Def:CRP}, the details of which will be dealt with in a separate paper, we are able to give a completely formal construction in the case when $A$ is a strict $\pi$-adically complete algebra over Lazard's universal $\pi$-ring, not to be confused with Lazard's universal ring~\cite[Th\'eor\`eme~II]{Laz55}.
	
	As described by Serre~\cite[Chapter II, \S5]{Ser79}, Lazard~\cite{Laz54} generalized the theory of Witt vectors to the now-classical equivalence of strict $p$-rings and perfect $\mathbf{F}_p$-algebras. Lazard's argument establishes an equivalence of categories by proving that the residue functor $U$ in Example~\ref{Ex:BabyTitling} is fully faithful~\cite[Chapter II, Proposition~10]{Ser79} and essentially surjective~\cite[Chapter II, Theorem~5]{Ser79}, without constructing an explicit quasiinverse. However Lazard does note that the `Wittsche Vektorrechnung' is an `Algorithmus' for calculating the universal $p^{-\infty}$-polynomials, and that these become honest polynomials after normalizing the indeterminates of the polynomial; Serre gives all the details~\cite[Chapter II, \S6]{Ser79}. Though omitted by Serre, Lazard~\cite[pp.~71--72]{Laz54} also dealt with the ramified case by considering the $\pi$-adically complete ring
	\begin{equation}\label{Eqn:LazardOriginalPiRing}
		\mathcal{L}:= \mathbf{Z}[\omega_i^{p^{-\infty}}\mid i \geq 1]\llbracket\pi\rrbracket\Bigg/\left(p-\sum_{i \geq 1} \omega_i \pi^i\right)
	\end{equation}
	with perfect residue ring
	\begin{equation*}
		\underline{\mathcal{L}} := \mathbf{F}_p[\omega_i\mid i \geq 1]^{\operatorname{pf}}.
	\end{equation*}
	We refer to $\mathcal{L}$ as \emph{Lazard's universal $\pi$-ring}. Given an object $O$ of $\mathsf{CRP}_\mathcal{L}$, a strict $\pi$-adically complete $O$-algebra will be called a \emph{Lazardian $O$-algebra}. An object $A$ of $\mathsf{CRP}_O$ is precisely a residually perfect Lazardian $O$-algebra.
	
	Lazard observes that the equivalence of categories $\mathsf{CRP}_{\mathbf{Z}_p} \to \mathsf{PerfAlg}_{\mathbf{F}_p}$ in Example~\ref{Ex:BabyTitling} generalizes to an equivalence $U: \mathsf{CRP}_{\mathcal{L}}\to \mathsf{PerfAlg}_{\underline{\mathcal{L}}}$ with $U:= -\otimes_{\mathcal{L}}\underline{\mathcal{L}}$. Following Serre~\cite[Chapter II, \S6]{Ser79}, we construct an explicit quasiinverse which we call the Lazardian Witt vector functor. There is only one obstacle in generalizing Serre \emph{mutatis mutandis}: the argument realizing the $p^{-\infty}$-polynomials that govern addition and multiplication as the reduction of Witt polynomials~\cite[Chapter II, Theorem~6]{Ser79}, in particular showing that those $p^{-\infty}$-polynomials are in fact honest polynomials, no longer works. However, a direct calculation, given in Proposition~\ref{Prop:DirectProofOfPolynomial}, is not difficult.
	
	By functoriality, the equivalence $U: \mathsf{CRP}_\mathcal{L}\to \mathsf{PerfAlg}_{\underline{\mathcal{L}}}$
	provided by the residue functor immediately generalizes to $\mathsf{CRP}_O \to \mathsf{PerfAlg}_{\underline{O}}$ for objects $O$ of $\mathsf{CRP}_{\mathcal{L}}$ with perfect residue ring $\underline{O}:= O/\pi O$. We use the Lazardian Witt vector functor to generalize Lazard's observation to Lazardian $O$-algebras, in particular allowing for imperfect residue ring. 
	\subsection{Main results}
	\begin{mainthm}\label{Thm:LazardianURP}
		
		Let $O$ be an object of $\mathsf{CRP}_\mathcal{L}$, and let $A$ be a Lazardian $O$-algebra. Given a solid diagram
		\begin{equation*}
			\begin{tikzcd}
				\mathsf{Alg}_{O}\arrow[from = d, shift right = 1ex, "W"'] &\mathsf{CRP}_{O}\arrow[l,hook']\arrow[ d, shift left = 1ex, "U"]&\mathsf{CRP}_A \arrow[d, shift left = 1ex, "U", dashed] \arrow[l,dashed]\\
				\mathsf{Alg}_{\underline{O}}\arrow[from = u, "\Delta"', shift right= 1ex, "\dashv"]&\mathsf{PerfAlg}_{\underline{O}} \arrow[l,hook']\arrow[u, "W", shift left = 1ex, "\dashv"']&\mathsf{PerfAlg}_{\Delta(A)^{\operatorname{pf}}}\arrow[u, "W", shift left = 1ex, "\dashv"',dashed]\arrow[l]
			\end{tikzcd}
		\end{equation*}
		in which the middle column forms an adjoint equivalence, there exist factorizations for $W$ and $U$ such that the right column forms an adjoint equivalence.
	\end{mainthm}
	
	This theorem, proven throughout Section~\ref{Sec:ThreeFunctors1} via an explicit construction of the claimed factorizations, is purely formal on the assumption that all data underlying the solid diagram exist. There are two main situations in which we invoke Theorem~\ref{Thm:LazardianURP}.
	
	The first is in the case of our Lazardian Witt vector functor construction. Existence is provided by  the classical \emph{residue functor} 
	\begin{equation*}
		U: \mathsf{CRP}_{\mathcal{L}}\to \mathsf{PerfAlg}_{\underline{\mathcal{L}}},
	\end{equation*}
	defined by $U := -\otimes_{\mathcal{L}}\underline{\mathcal{L}}$; the \emph{Lazardian Witt vector functor}
	\begin{equation*}
		W: \mathsf{Alg}_{\underline{\mathcal{L}}}\to \mathsf{Alg}_{\mathcal{L}},
	\end{equation*}
	defined in Theorem~\ref{Theorem:LazardianWittVectorConstruction}; the required adjoint equivalence $W\dashv U$, established in Section~\ref{Sec:ExplicitAdjEq}; the \emph{Lazardian jet algebra functor} 
	\begin{equation*}
		\Delta: \mathsf{Alg}_{\mathcal{L}}\to\mathsf{Alg}_{\underline{\mathcal{L}}},
	\end{equation*} 
	defined in Definition~\ref{Defn:LazardianJetAlgebra} and Proposition~\ref{Prop:jetAlgebraFunctoriality}; and the required adjunction $\Delta \dashv W$, established in the remainder of Section~\ref{Sec:LazardianJetAlgebras}. In short, these data provide the following existence statement.

	\begin{mainthm}\label{Thm:LazardianWittVectors}
		Let $U: \mathsf{CRP}_{\mathcal{L}} \to \mathsf{PerfAlg}_{\underline{\mathcal{L}}}$ be the residue functor. The Lazardian Witt vector functor $W: \mathsf{Alg}_{\underline{\mathcal{L}}} \to \mathsf{Alg}_{\mathcal{L}}$ and the Lazardian jet algebra functor $\Delta: \mathsf{Alg}_{\mathcal{L}} \to \mathsf{Alg}_{\underline{\mathcal{L}}} $ fit into the diagram
		\begin{equation*}
			\begin{tikzcd}
				\mathsf{Alg}_{\mathcal{L}}\arrow[from = d, shift right = 1ex, "W"'] &\mathsf{CRP}_{\mathcal{L}}\arrow[l,hook']\arrow[ d, shift left = 1ex, "U"]\\
				\mathsf{Alg}_{\underline{\mathcal{L}}}\arrow[from = u, "\Delta"', shift right= 1ex, "\dashv"]&\mathsf{PerfAlg}_{\underline{\mathcal{L}}} \arrow[l,hook']\arrow[u, "W", shift left = 1ex, "\dashv"']
			\end{tikzcd}
		\end{equation*}
		in which the right column forms an adjoint equivalence.
	\end{mainthm}
	
	The second case in which we invoke  Theorem~\ref{Thm:LazardianURP} is when $O$ is the integer ring of a finite extension of $\mathbf{Q}_p$,  or alternatively when $O= \mathbf{F}_p\llbracket\pi\rrbracket$. When $O$ is the integer ring of a finite extension of $\mathbf{Q}_p$ and $\pi \in O$ is a uniformizer, we take $\Delta$ to be the free $\delta_{\pi}$-ring functor, and we take $W$ to be the $\pi$-typical ramified Witt vector functor. We may also include the classical case of $O:= \mathbf{Z}$ if we distinguish $O$ from its completion $\widehat{O}$ with respect to a uniformizer. In other words, we work with the diagram
	\begin{equation*}
		\begin{tikzcd}
			\mathsf{Alg}_O\arrow[from = d, shift right = 1ex, "W"']\arrow[r, shift right = 1ex, "\perp"] & \mathsf{Alg}_{\widehat{O}}\arrow[l, shift right = 1ex]& \mathsf{CRP}_{\widehat{O}}\arrow[l,hook']\arrow[ d, shift left = 1ex, "U"] \\
			\mathsf{Alg}_O \arrow[from = u, "\Delta"', shift right= 1ex, "\dashv"]\arrow[r, shift right = 1ex, "\top"]& \mathsf{Alg}_{\underline{O}}\arrow[l, shift right = 1ex]\arrow[u, "W",dashed] & \mathsf{PerfAlg}_{\underline{O}}.\arrow[l,hook']\arrow[u, "W", shift left = 1ex, "\dashv"']
		\end{tikzcd}
	\end{equation*}

	 When $O= \mathbf{F}_p\llbracket\pi\rrbracket$, we take $\Delta:= \operatorname{HS}(-)\otimes_{\operatorname{HS}(\mathbf{F}_p\llbracket\pi\rrbracket)}\mathbf{F}_p$, where $\operatorname{HS}$ is the functor of Hasse-Schmidt derivations in the sense of Vojta~\cite{Voj07}, alternatively known as jet algebras~\cite{LM09,MP24}, and we take $W := (-)\llbracket\pi\rrbracket$ to be the power series functor. 
	
	Given a Lazardian $O$-algebra $A$ as in the hypotheses of Theorem~\ref{Thm:LazardianURP}, we obtain a construction
	\begin{equation*}
		k^u = \Delta(A)^{\operatorname{pf}}
	\end{equation*}
	of the residue ring that does not require a choice of $p$-basis. However, the data of an $O$-algebra structure implies a choice of uniformizer, and moreover the functor $W$ from which we obtain the universal residual perfection
	\begin{equation*}
		A^u = W(\Delta(A)^{\operatorname{pf}})
	\end{equation*} 
	depends on our choice of uniformizer. Since we may take $O:= \mathbf{F}_p\llbracket\pi\rrbracket$, this construction covers the equal characteristic case in its entirety. However the mixed characteristic case does not appear to admit any simple general description when $A$ does not admit any $\mathcal{L}$-algebra structure.

	For the particular case of $O := \mathbf{F}_p\llbracket\pi\rrbracket$, we have an explicit calculation.
	\begin{corollary*}
		Let $k$ be an $\mathbf{F}_p$-algebra. The universal residual perfection of $ k\llbracket\pi\rrbracket$ is 
		\begin{equation*}
			(k\llbracket\pi\rrbracket)^u = \operatorname{HS}(k)^{\operatorname{pf}}\llbracket\pi\rrbracket.
		\end{equation*} 
	\end{corollary*}
	\begin{proof}
		See Corollary~\ref{Cor:ExplicitPowerSeriesCalc}.
	\end{proof}
	This corollary gives a canonical description of the universal residual perfection for a power series ring over an arbitrary $\mathbf{F}_p$-algebra, mainly because the choice of uniformizer is already implicit. In the case of equal characteristic complete discrete valuation rings, it only shifts the choice of $p$-basis and uniformizer from the construction of $k_T$ and an infinitesimal thickening of $k_T^{\operatorname{pf}}$ to a choice of isomorphism $k\llbracket\pi\rrbracket\xrightarrow{\sim}A$ in service of the Cohen structure theorem. 
	
	\subsection{Ramified Witt vectors in the literature}
	The earliest commonly-cited works are that of Ditters~\cite{Dit75} and Drinfeld~\cite{Dri76}. Hazewinkel~\cite{Haz80} nonexhaustively credits these two authors along with unpublished work of J.~Casey for the finite residue field untwisted case. Hazewinkel then introduces twisted ramified Witt vectors for complete discrete valuation rings admitting a Frobenius lift, without any restriction on the residue field. Some related subsequent works~\cite{Che18, Mat19} adopt the same setup as Hazewinkel, while most others~\cite{Bor11,ACZ16,  PT16,Sch17,Ver17,FF18,BC20,BMS23, BPS23,Che23} appear to invariably deal with the original case of an integer ring of a finite extension of $\mathbf{Q}_p$. So the only question is to what extent does the condition of being Lazardian overlap with Hazewinkel's setup of a complete discrete valuation ring admitting a Frobenius lift. The author of this paper is unable to provide any meaningful answer to this question.
	
	\subsection{Conventions}
	Throughout this paper, for all $I$-adically complete rings $A$ such that $k:= A/I$ is a perfect $\mathbf{F}_p$-algebra, we let $[-]: k \to A$ denote the unique multiplicative section~\cite[Chapter~II, Proposition~8]{Ser79}. Occasionally we will encounter a situation involving an $I$-adically complete ring $A$ along with an $A$-algebra $R$ that is not necessarily complete. The notation $\widehat{R}$, or alternatively $R^{\widehat{~}}$, will be used to denote $IR$-adic completion.
	
	\subsection{Acknowledgments}
	The author thanks James Borger for the comments on drafts and for the pertinacious insistence that the moduli of residual perfection should have a description in terms of jets. This research was supported by an Australian Government Research Training Program.
	\section{Lazardian Witt vectors}

	We define a slight generalization $\mathcal{L}_{m,q}^{(t)}$ of Lazard's universal $\pi$-ring, along with its residue ring $\underline{\mathcal{L}_m}$, in order to obtain the \emph{arithmetic polynomials}
	\begin{equation*}
		Q^{+(t)}_{n,q}, Q^{\times(t)}_{n,q} \in \mathbf{F}_p[\omega_1^{q^t},\dots,\omega_n^{q^t};X_0,\dots,X_n;Y_0,\dots,Y_n]
	\end{equation*} 
	that furnish affine $(m+1)$-space $\mathbf{A}^{m + 1}_{\underline{\mathcal{L}_m}}$ with the structure of an $\mathcal{L}_{m,q}^{(t)}$-algebra scheme. The fact that the $Q_{n,q}^{+ (t)}, Q^{\times (t)}_{n,q}$ are polynomials is proven in Proposition~\ref{Prop:DirectProofOfPolynomial}, while the construction of the Lazardian Witt vector functor is given in Theorem~\ref{Theorem:LazardianWittVectorConstruction}. We end the section by discussing additional structures such as the Verschiebung, the multiplicative section, and the Frobenius operator, in the expository style of Serre~\cite[p.~42-44]{Ser79}.
	
	\begin{definition}
		Fix a prime $p > 0$. Let $t \in \mathbf{Z}$, let $m \in \mathbf{N}\cup \{\infty\}$, and  let $q>1$ be some $p$th power. Define \emph{Lazard's universal $\pi$-ring} by 
		\begin{equation*}
			\mathcal{L}^{(t)}_{m,q} := \mathbf{Z}[\omega_{i}^{q^{-\infty}}\mid 1 \leq i \leq m]\llbracket\pi\rrbracket\Bigg/\left(p - \sum_{i = 1}^m \omega_{i }^{q^{t}}\pi^i, \pi^{m + 1}\right),
		\end{equation*}
		and let 
		\begin{equation*}
			\underline{\mathcal{L}_{m}}:= \mathbf{F}_p[\omega_i \mid 1\leq i \leq m]^{\operatorname{pf}}
		\end{equation*}
		be its residue ring. When $t,m,q$ are not mentioned, we will take these values to be $t = 0$, $m = \infty$, and $q = p$, as in Lazard's original definition.
	\end{definition}

	In the following, we will need to make a distinction between the Frobenius endomorphism of $\underline{\mathcal{L}}[X_i^{p^{-\infty}};Y_i^{p^{-\infty}}\mid i \in \mathbf{N}]$ and the Frobenius endomorphism applied only on the coefficient ring, which is base change by the Frobenius endomorphism. We will denote the Frobenius endomorphism applied on a polynomial by $f(X,Y)^q$, and we will denote the Frobenius endomorphism applied only on the coefficients by $f^q$. Under this convention, we have $f(X,Y)^q = f^q(X^q, Y^q)$. 
	\begin{proposition}\label{Prop:DirectProofOfPolynomial}
		Let $\ast$ be either operation $+,\times$ and let $Q_{n,q}^{\ast (t)} \in \underline{\mathcal{L}}[X_i^{p^{-\infty}};Y_i^{p^{-\infty}}\mid i \in \mathbf{N}]$ be the $p^{-\infty}$-polynomials determined by the equations
		\begin{equation*}
			\sum_i X_i^{q^{-i}}\pi^i \ast \sum_i Y_i^{q^{-i}}\pi^i  = \sum_i [Q_{i,q}^{\ast(t) }(X,Y)^{q^{-i}}] \pi^i
		\end{equation*} 
		in the ring $\mathcal{L}_{ q}^{(t)}[X_i^{p^{-\infty}};Y_i^{p^{-\infty}}\mid i \in \mathbf{N}]^{\widehat{~}}$. 
		\begin{enumerate}[\normalfont (i)]
			\item Each $Q_{n,q}^{\ast(t)}$ is contained in the subring $\mathbf{F}_p[\omega^{q^t}_1,\dots,\omega^{q^t}_n;X_0,\dots,X_n;Y_0,\dots,Y_n]$. 
			\item The base change 
			\begin{equation*}
				\begin{tikzcd}
						\underline{\mathcal{L}}[X_i^{p^{-\infty}}; Y_i^{p^{-\infty}}\mid i \in \mathbf{N}] \arrow[r]& \	\underline{\mathcal{L}}[X_i^{p^{-\infty}}; Y_i^{p^{-\infty}}\mid i \in \mathbf{N}]\\
						\underline{\mathcal{L}} \arrow[u]\arrow[r,"\text{\normalfont Frobenius}"]&	\underline{\mathcal{L}}\arrow[u]
				\end{tikzcd}
			\end{equation*}
			sends $Q_{n,q}^{\ast(t)}$ to $Q_{n,q}^{\ast (t + 1)}$; in other words $Q^{\ast(t)q}_{n,q} = Q^{\ast (t + 1)}_{n,q}$. 
		\end{enumerate}
	\end{proposition}
	\begin{proof}
		(i) Note that $Q_{0,q}^{\ast(t)} = X_0 \ast Y_0 \in \mathbf{F}_p[X_0;Y_0]$ is a polynomial. Assume that there exists $n \geq 1$ such that 
		\begin{equation*}
			Q_{i,q}^{\ast(t)} \in \mathbf{F}_p[\omega^{q^t}_1,\dots,\omega^{q^t}_i;X_0,\dots,X_i;Y_0,\dots,Y_i]
		\end{equation*} 
		for all $i < n$. Define $\mathcal{P}_n\subset\mathcal{L}_{ q}^{(t)}[X_i^{p^{-\infty}};Y_i^{p^{-\infty}}\mid i \in \mathbf{N}]/(\pi^{n + 1})$ as the image
		\begin{equation*}
			\begin{tikzcd}
				\displaystyle\frac{\mathbf{Z}[\omega_i^{q^t};X_i^{q^{-n}};Y_i^{q^{-n}}\mid i \leq n]\llbracket\pi\rrbracket}{\left(p-\sum_{i \leq n}\omega_i^{q^t}\pi^i,\pi^{n +1}\right)}\arrow[r,"\sim"]\arrow[rd,tail] & \mathcal{P}_n\arrow[d,hook] \\
				&\mathcal{L}_{ q}^{(t)}[X_i^{p^{-\infty}};Y_i^{p^{-\infty}}\mid i \in \mathbf{N}]/(\pi^{n + 1}).
			\end{tikzcd}
		\end{equation*} 
		In particular the image $\underline{\mathcal{P}_n} \subset \underline{\mathcal{L}}[X_i^{p^{-\infty}};Y_i^{p^{-\infty}}\mid i \in \mathbf{N}]$ satisfies 
		\begin{equation*}
			(\underline{\mathcal{P}_n})^{q^n} \subset \mathbf{F}_p[\omega^{q^t}_1,\dots,\omega^{q^t}_n;X_0,\dots,X_n;Y_0,\dots,Y_n]. 
		\end{equation*}

		To calculate $Q_{n,q}^{\ast (t)}$, we consider the congruence
		\begin{equation*}
			\sum_{i = 0}^{n } X_i^{q^{-i}}\pi^i \ast \sum_{i = 0}^{n } Y_i^{q^{-i}}\pi^i  \equiv \sum_{i = 0}^{n} [Q_{i,q}^{\ast(t) }(X,Y)^{q^{-i}}] \pi^i \bmod \pi^{n + 1}.
		\end{equation*} 
		Observe that
		\begin{equation*}
				\sum_{i = 0}^{n } X_i^{q^{-i}}\pi^i \ast \sum_{i = 0}^{n } Y_i^{q^{-i}}\pi^i  \in \mathcal{P}_n.
		\end{equation*} 
		To calculate the multiplicative lift $[Q_{i,q}^{\ast (t)}(X,Y)^{q^{-i}}]$ for $i < n$, we may take the $q^n$th power of any lift of $Q_{i,q}^{\ast (t)}(X,Y)^{q^{-(i + n)}}$. Choosing lifts
		\begin{equation*}
			P_{i,q}^{\ast (t)} \in \mathbf{Z}[\omega^{q^t}_1,\dots,\omega^{q^{t}}_i;X_0,\dots,X_i;Y_0,\dots,Y_i] 
		\end{equation*}
		of $Q_{i,q}^{\ast(t)}$, we obtain
		\begin{equation*}
		P_{i,q}^{\ast (t)}(\omega^{q^{-(i + n)}},X^{q^{-(i + n)}},Y^{q^{-(i + n)}})^{q^n}\equiv [Q_{i,q}^{\ast (t)}(X,Y)^{q^{-i}}]\bmod \pi^{n + 1}
		\end{equation*}
		for all $i < n$.

		Write $P_{i,q}^{\ast (t)} = x_1 + \dots + x_\alpha$ where each $x_1,\dots,x_\alpha$ is a distinct monomial in $\omega, X, Y$. We have 
		\begin{equation*}
			P_{i,q}^{\ast (t)}(\omega,X,Y)^{q^n} = \sum_{k_1 + \dots + k_\alpha = q^n} \binom{q^n}{k_1,\dots,k_\alpha} x_1^{k_1} \dots x_\alpha^{k_\alpha}
		\end{equation*} 
		by the multinomial theorem. Suppose that $e \geq 1$ is such that $q = p^e$. Define the $q$-adic valuation by $v_q(x) := v_p(x)/e$, where $v_p$ is the ordinary $p$-adic valuation. If there exists some $k_j$ with $v_q(k_j) < i$, then 
		\begin{equation*}
			v_q\binom{q^n}{k_1,\dots,k_\alpha} \geq v_q\binom{q^n}{k_j} = n-v_q(k_j) > n-i.
		\end{equation*}
		Hence $\binom{q^n}{k_1,\dots,k_\alpha}\pi^i \equiv 0 \bmod \pi^{n + 1}$, and, after multiplying by $\pi^i$, it suffices to sum over the tuples $(k_1,\dots,k_\alpha)$ such that $v_q(k_j) \geq i$ for all $j$. This implies
		\begin{equation*}
		[Q_{i,q}^{\ast (t)}(X,Y)^{q^{-i}}]\pi^i = 	P_{i,q}^{(t)}(\omega^{q^{-(i + n)}}, X^{q^{-(i + n)}}, Y^{q^{-(i + n)}})^{q^n}\pi^i \in \mathcal{P}_n
		\end{equation*}
		for all $i < n$. Therefore 
		\begin{equation*}
			[Q_{n,q}^{\ast(t)}(X,Y)^{q^{-n}}]\pi^{n } = 	\sum_{i = 0}^{n } X_i^{q^{-i}}\pi^i \ast \sum_{i = 0}^{n } Y_i^{q^{-i}}\pi^i  - \sum_{i = 0}^{n-1} [Q_{i,q}^{\ast(t) }(X,Y)^{q^{-i}}] \pi^i  \in \mathcal{P}_n.
		\end{equation*}
		We deduce that $Q_{n,q}^{\ast(t)}$ is in $\mathbf{F}_p[\omega_1^{q^t},\dots,\omega_n^{q^t};X_0,\dots,X_n;Y_0,\dots,Y_n]$ after canceling by $\pi^n$ and taking $q^n$th powers. 
		
		(ii) In $\mathcal{L}^{(t)}_q[X_i^{p^{-\infty}}; Y_i^{p^{-\infty}}]^{\widehat{~}}$, we have
		\begin{equation*}
			\sum_{i} X_i^{q^{-i}} \pi^i \ast \sum_{i} Y_i^{q^{-i}} \pi^i = \sum_{i}[Q_{i,q}^{\ast(t)}(X,Y)^{q^{-i}}]\pi^i.
		\end{equation*}
		Since the lifted Frobenius operator $\mathcal{L}^{(t)}_q[X_i^{p^{-\infty}}; Y_i^{p^{-\infty}}]^{\widehat{~}} \to \mathcal{L}^{(t+1)}_q[X_i^{p^{-\infty}}; Y_i^{p^{-\infty}}]^{\widehat{~}}$ is the Frobenius endomorphism on coordinates, we obtain
		\begin{equation*}
			\sum_{i}[Q_i^{\ast(t + 1)}(X^q,Y^q)^{q^{-i}}]\pi^i = \sum_{i}[Q_i^{\ast(t) q}(X^q,Y^q)^{q^{-i}}]\pi^i.
		\end{equation*}
		Uniqueness of coordinates gives
		\begin{equation*}
			Q_i^{\ast(t + 1)1/q}(X,Y)^q = Q_i^{\ast(t+1)}(X^q,Y^q) = Q_i^{\ast(t) q}(X^q,Y^q) = Q_i^{\ast (t)}(X,Y)^q.
		\end{equation*}
		Thus $Q_i^{\ast(t + 1)} = Q_i^{\ast(t) q}$. 
	\end{proof}
	\begin{example}
		We have 
		\begin{equation*}
			Q_0^{+} = X_0 + Y_0,\quad  Q_1^+ = X_1 + Y_1 + (p/\pi)^q C_q(X_0,Y_0),
		\end{equation*} 
		where $C_q \in \mathbf{Z}[X,Y]$ is defined implicitly by $p C_q = X^q + Y^q - (X +Y)^q$ in $\mathbf{Z}[X,Y]$. On the other hand, the first Witt polynomial for addition is
		\begin{equation*}
			S_1 = X_1 + Y_1 + (p/\pi) C_q(X_0,Y_0).
		\end{equation*}
		Therefore $\overline{S_1} \neq Q_1^+$, unlike the classical case of ramified Witt vectors. However, we observe that $S_1$ reduces to $Q_1^+$ if we specialize to a residue ring in which the first coefficient $\omega_1 = \overline{p/\pi}$ for $p$ is an element of $\mathbf{F}_q$, as would be the case for ramified Witt vectors with uniformizer $\pi$ in a finite extension of $\mathbf{Q}_p$ with residue field $\mathbf{F}_q$.

		We have 
		\begin{equation*}
			Q_0^\times = X_0Y_0,\quad  Q_1^\times = X_0^qY_1 + X_1 Y_0^q.
		\end{equation*} 
		The first Witt polynomial for multiplication is
		\begin{equation*}
			P_1 = X_0^qY_1 + X_1Y_0^q + \pi X_1Y_1.
		\end{equation*}
		 Even though $\overline{P_1} = Q_1^\times$, we observe that the `naive' lift of $Q_1^\times$ is not equal to $P_1$.
	\end{example}
	\begin{theorem}\label{Theorem:LazardianWittVectorConstruction}
		Let $m \in \mathbf{N}\cup \{\infty\}$, and let  $k$ be an $\underline{\mathcal{L}_m}$-algebra. For $a,b \in k^{m + 1}$, define
		\begin{equation*}
			a + b := (Q^{+(t)}_{i,q}(a,b))_i, \quad ab := (Q^{\times (t)}_{i,q}(a,b))_i.
		\end{equation*}
		These binary operations make the set $k^{m + 1}$ into an $\mathcal{L}_{m,q}^{(t)}$-algebra $W^{(t)}_{m,q}(k)$ with structure map $\omega_j \mapsto (\omega_j,0,0,\dots)$ and $\pi \mapsto (0,1,0,\dots)$.  The $\mathcal{L}_{m,q}^{(t)}$-algebras  $W^{(t)}_{m,q}(k)$ define a functor 
		\begin{equation*}
			\begin{tikzcd}
				\mathsf{Alg}_{\mathcal{L}_{m,q}^{(t)}} & \mathsf{CRP}_{\mathcal{L}_{m,q}^{(t)}}\arrow[l,hook']\\
				\mathsf{Alg}_{\underline{\mathcal{L}_m}}\arrow[u, "W^{(t)}_{m,q}"]&\mathsf{PerfAlg}_{\underline{\mathcal{L}_m}}\arrow[l,hook']\arrow[u, "W^{(t)}_{m,q}"']
			\end{tikzcd}
		\end{equation*} 
		called the \emph{Lazardian Witt vector functor}, furnishing  $\mathbf{A}_{\underline{\mathcal{L}_m}}^{m + 1}$ with the structure of an $\mathcal{L}_{m,q}^{(t)}$-algebra scheme. 
	\end{theorem}
	\begin{proof}
		Choose a quasiinverse $H: \mathsf{PerfAlg}_{\underline{\mathcal{L}_m}}\to \mathsf{CRP}_{\mathcal{L}_{m,q}^{(t)}}$ to the equivalence given by the residue functor. If $k$ is perfect, then the function
		\begin{equation*}
			\begin{tikzcd}[row sep = 0.5ex]
				W^{(t)}_{m,q}(k) \arrow[r]& H(k) \\
				(\alpha_i)_i \arrow[r,maps to]& \sum_i [\alpha_i^{q^{-i}}]\pi^i
			\end{tikzcd}
		\end{equation*} 
		is additive and multiplicative by definition of the binary operations on $W^{(t)}_{m,q}(k)$. Strictness of $H(k)$ implies that the function is injective. Perfectness of $k$ and $\pi$-adic completeness of $H(k)$ imply that the function is surjective. Thus the function $W_{m,q}^{(t)}(k) \to H(k)$ is a ring isomorphism. The structure map $\mathcal{L}_{m,q}^{(t)} \to H(k)$ composed with the inverse of this isomorphism gives an $\mathcal{L}_{m,q}^{(t)}$-algebra structure on $W^{(t)}_{m,q}(k)$ satisfying $\omega_j \mapsto (\omega_j,0,0,\dots)$ and $\pi \mapsto (0,1,0,\dots)$. Once the claim holds for $k$, it holds for all $\underline{\mathcal{L}_m}$-subalgebras and quotients of $k$. But every $\underline{\mathcal{L}_m}$-algebra is a quotient of  $\underline{\mathcal{L}_m}[X_\lambda \mid \lambda \in \Lambda] \subset \underline{\mathcal{L}_m}[X_\lambda\mid \lambda \in \Lambda]^{\operatorname{pf}}$ for some indexing set $\Lambda$. 
	\end{proof}
	\subsection{Additional structures}
	The $p$-typical Witt vector functor comes with the additional data of a multiplicative lift, a Frobenius operator, and a Verschiebung operator. The Lazardian Witt vector functor also has these additional data with some modifications.

	Unlike the $p$-typical Witt vector functor, we do not have a Frobenius endomorphism because the coefficients of $Q_{i,q}^{\ast(t)}$ are, in general, not fixed by any iterate of the Frobenius map. Instead, we have a Frobenius operator $W^{(t)}_{m,q} \to W^{(t + 1)}_{m,q}$. Given that the Frobenius operator shifts from $t$ to $t + 1$, we would expect that the Verschiebung does the opposite if we are to obtain the classical formula $\pi = FV$. This happens to be the case.
	\begin{proposition}\label{Prop:tuplesToSum}
		Let $t \in \mathbf{Z}$, let $m \in \mathbf{N}\cup \{\infty\}$, and let $q>1$ be a $p$th power. The assignment 
		\begin{equation*}
			\begin{tikzcd}[row sep = .5ex]
				\mathbf{A}^1_{\underline{\mathcal{L}_m}} = W^{(t)}_{0,q} \arrow[r,"{[-]}"]&  W^{(t)}_{m,q}\\
				 \alpha \arrow[r,maps to]&  (\alpha,0,\dots)
			\end{tikzcd}
		\end{equation*}
		defines a natural transformation of multiplicative monoid schemes over $\underline{\mathcal{L}_m}$. For all $ (\beta_i)_i \in W^{(t)}_{m,q}(k)$, one has $[\alpha](\beta_i)_i= (\alpha^{q^i}\beta_i)_i$. If $k$ is perfect, then $(\beta_i)_i = \sum_i [\beta_i^{q^{-i}}]\pi^i$ in the $\pi$-adically complete $\mathcal{L}_{m,q}^{(t)}$-algebra $W^{(t)}_{m,q}(k)$.
	\end{proposition}
	\begin{proof}
		It suffices to calculate this on perfect $\underline{\mathcal{L}_m}$-algebras, in which case multiplicativity follows from the usual theory of the multiplicative section, and the calculations $[\alpha](\beta_i)_i = (\alpha^{q^i}\beta_i)_i$ and $(\beta_i)_i = \sum_i [\beta_i^{q^{-i}}]\pi^i$ follow from a choice of isomorphism $W^{(t)}_{m,q}(k) \xrightarrow{\sim}H(k)$ into an object $H(k)$ of $\mathsf{CRP}_{\mathcal{L}_{m,q}^{(t)}}$ with perfect residue ring $k$. 
	\end{proof}
	\begin{proposition}
		Let $t \in \mathbf{Z}$, let $m \in \mathbf{N}\cup \{\infty\}$, and let $q > 1$ be a $p$th power.  The Frobenius operator
		\begin{equation*}
			\begin{tikzcd}[row sep = .5ex]
				W^{(t)}_{m,q}\arrow[r]&W^{(t+1)}_{m,q}\\ (\alpha_0,\alpha_1,\dots) \arrow[r,maps to]& (\alpha_0^q,\alpha_1^q,\dots)
			\end{tikzcd}
		\end{equation*} 
		defines a morphism of ring schemes over $\underline{\mathcal{L}_m}$.
	\end{proposition}
	\begin{proof}
		The fact that the Frobenius operator is a ring morphism on each $\underline{\mathcal{L}}_m$-algebra follows from $Q_{i,q}^{\ast(t)q} = Q_{i,q}^{\ast(t + 1)}$. Naturality follows from the Frobenius operator commuting with ring morphisms.
	\end{proof}
	\begin{proposition}
		Let $t \in \mathbf{Z}$, let $m \in \mathbf{N}\cup \{\infty\}$, and let $q > 1$ be a $p$th power.
		\begin{enumerate}[\normalfont (i)]
			\item The Verschiebung 
			\begin{equation*}
				\begin{tikzcd}[row sep = .5ex]
					W^{(t)}_{m,q}\arrow[r, "V"]& W^{(t-1)}_{m+1,q}\\
					(\alpha_0,\alpha_1,\alpha_2,\dots) \arrow[r, maps to ]&(0,\alpha_0,\alpha_1,\dots)
				\end{tikzcd}
			\end{equation*} 
			defines a monomorphism of additive group schemes over $\underline{\mathcal{L}_m}$.
			\item  If $F: W^{(t)}_{m + 1,q}\to W^{(t+1)}_{m + 1,q}$ denotes the Frobenius, then $\pi \iota = FV$, where $\iota: W^{(t)}_{m,q} \to W^{(t)}_{m + 1,q}$ is the injection defined by $\iota(\alpha_0,\dots,\alpha_m) = (\alpha_0,\dots,\alpha_m,0)$.  
			\item Let $m, r\in \mathbf{N}$. We have an exact sequence
			\begin{equation*}
				\begin{tikzcd}
					0 \arrow[r]& W^{(t)}_{m,q}  \arrow[r, "V^r"]& W^{(t-r)}_{m + r,q} \arrow[r] & W^{(t-r)}_{r-1,q} \arrow[r]&0
				\end{tikzcd}
			\end{equation*}
			of additive group schemes.
		\end{enumerate}
	\end{proposition}
	\begin{proof}
		(i) and (ii) Naturality of $V$ is immediate, so it remains to  show additivity and the formula $\pi\iota = FV$. It suffices to establish these properties for perfect $\underline{\mathcal{L}_m}$-algebras. In this case additivity of $V$ follows from the formula $\pi\iota = FV$ since $F$ is an isomorphism for perfect $k$. 
		
		By Proposition~\ref{Prop:tuplesToSum}, setting $\alpha_{-1}:= 0$, we may write
		\begin{equation*}
			FV(\alpha_i)_i = (\alpha_{i-1}^q)_i = \sum_i [\alpha_i^{q^{-i}}]\pi^{i + 1} = \pi \iota(\alpha_i)_i.
		\end{equation*}
		Thus $\pi \iota = FV$ and the Verschiebung is additive.

		(iii) The calculations
		\begin{equation*}
			(\alpha_0,\dots,\alpha_m)\mapsto \begin{matrix}
				(0,\dots,0,\alpha_0,\dots,\alpha_m) \\
				(\beta_0,\dots,\beta_{ m + r})
			\end{matrix} \mapsto \begin{matrix}
			0 \\
			(\beta_0,\dots,\beta_{r-1}),
			\end{matrix}
		\end{equation*}
		establish that the kernel of truncation is the image of $V^r$. Injectivity of $V$ and surjectivity of truncation are immediate by calculating on coordinates.
	\end{proof}
	\section{Explicit adjoint equivalence}\label{Sec:ExplicitAdjEq}
	To verify the $\mathcal{L}_{m,q}^{(t)}$-algebra scheme structure of $W := W^{(t)}_{m,q}$ in the proof of Theorem~\ref{Theorem:LazardianWittVectorConstruction}, we appealed to the existence of a quasiinverse to the residue functor $U: \mathsf{CRP}_{\mathcal{L}^{(t)}_{m,q}}\to \mathsf{PerfAlg}_{\underline{\mathcal{L}_m}}$, where $U:= -\otimes_{\mathcal{L}_{m,q}^{(t)}}\underline{\mathcal{L}_m}$. The full faithfulness and essential surjectivity of $U$ is the classical result of Lazard~\cite[pp.~71-72]{Laz54}. Now that we know that $W$ is an $\mathcal{L}_{m,q}^{(t)}$-algebra scheme, the goal for this section is to construct the adjoint equivalence in the right column of the diagram
	\begin{equation*}
		\begin{tikzcd}
			\mathsf{Alg}_{\mathcal{L}_{m,q}^{(t)}}\arrow[from = d, "W"'] &\mathsf{CRP}_{\mathcal{L}_{m,q}^{(t)}}\arrow[l,hook']\arrow[ d, shift left = 1ex, "U",dashed]\\
			\mathsf{Alg}_{\underline{\mathcal{L}_m}}&\mathsf{PerfAlg}_{\underline{\mathcal{L}_m}}, \arrow[l,hook']\arrow[u, "W", shift left = 1ex, "\dashv"',dashed]
		\end{tikzcd}
	\end{equation*}
	which is one of the two adjunctions comprising Theorem~\ref{Thm:LazardianWittVectors}.

	We start by constructing the unit of the adjoint equivalence.
	\begin{proposition}
		Let $k$ be a perfect $\underline{\mathcal{L}_m}$-algebra. The functions $\eta_k: k \to UW(k)$ defined by $\alpha \mapsto \overline{(\alpha,0,\dots)}$ determine a natural isomorphism 
		\begin{equation*}
			\eta: 1_{\mathsf{PerfAlg}_{\underline{\mathcal{L}_m}}} \to UW. 
		\end{equation*}
	\end{proposition}
	\begin{proof}
		Since $k$ is perfect, we have
		\begin{equation*}
			(0,\alpha_1,\alpha_2,\dots) = FV(\alpha_1^{1/q},\alpha_2^{1/q},\dots) = \pi(\alpha_1^{1/q},\alpha_2^{1/q},\dots).
		\end{equation*}
		On the other hand $\pi = (0,1,0,\dots)\in \ker(W(k)\to k)$, so $\pi W(k) = \ker(W(k)\to k)$.
		Therefore the truncation morphism $W(k) \to k$ admits the factorization
		\begin{equation*}
			\begin{tikzcd}[row sep = .5ex]
				UW(k)\arrow[r, "\sim"]&k\\ \overline{(\alpha_0,\alpha_1,\dots)} \arrow[r,maps to]& \alpha_0
			\end{tikzcd}
		\end{equation*} 
		which we observe is the set-theoretic inverse to $\eta_k$. Therefore $\eta_k$ is an isomorphism of rings. To see that $\eta_k$ is a $\mathsf{PerfAlg}_{\underline{\mathcal{L}_m}}$-algebra map, we consider the diagram in which only commutativity of the bottom triangle
		\begin{equation*}
			\begin{tikzcd}
				\mathcal{L}_{m,q}^{(t)} \arrow[rr]\arrow[dd]& & W(k)\arrow[ld]\arrow[dd] \\
				&k\\
				\underline{\mathcal{L}_m} \arrow[ru]\arrow[rr]&&UW(k)\arrow[lu, "\sim"']
			\end{tikzcd}
		\end{equation*}
		is in question. A diagram chase through the outer rectangle shows that the bottom triangle commutes because $\mathcal{L}_{m,q}^{(t)} \to \underline{\mathcal{L}_m}$ is a surjection and hence an epimorphism. Naturality in $k$ is immediate since $W$ acts on morphisms coordinatewise. 
	\end{proof}
	
	\begin{remark}
		We used the fact that $k$ is perfect to obtain $\pi W(k) = \ker(W(k) \to k)$. Conversely, if  $\pi W(k) = \ker(W(k) \to k)$, then $(0,\alpha,0,\dots) \in \pi W(k)$ for all $\alpha \in k$. On the other hand
		\begin{equation*}
			\pi (\beta_0,\beta_1,\dots) = FV(\beta_0,\beta_1,\dots) = (0, \beta_0^q,\dots).
		\end{equation*}
		Thus the containments $k \supset k^p \supset k^q$ are equalities and $k$ is perfect. So the adjunction only holds for perfect $\underline{\mathcal{L}_m}$-algebras, despite $W$ being defined on all $\underline{\mathcal{L}_m}$-algebras.
	\end{remark}

	We now construct the counit.
	\begin{proposition}\label{Prop:WUcounit}
		Let $A$ be an object of $\mathsf{CRP}_{\mathcal{L}_{m,q}^{(t)}}$. The $\mathcal{L}_{m,q}^{(t)}$-algebra morphisms $\varepsilon_A: WU(A) \to A$ defined by $(\alpha_0,\alpha_1,\dots) \mapsto \sum_i [\alpha_i^{q^{-i}}]\pi^i$ determine a natural isomorphism 
		\begin{equation*}
			\varepsilon: WU \to 1_{\mathsf{CRP}_{\mathcal{L}^{(t)}_{m,q}}}.
		\end{equation*}
	\end{proposition}
	\begin{proof}
		The fact that $\varepsilon_A$ is an $\mathcal{L}_{m,q}^{(t)}$-algebra morphism is by definition of the algebra structure on $W$ in terms of the arithmetic polynomials. Let $\Phi: A \to A'$ be an $\mathcal{L}_{m,q}^{(t)}$-morphism. Compatibility of morphisms with the multiplicative section implies that $\Phi[\alpha] = [U(\Phi)(\alpha)]$. Continuity of $\Phi$  implies that $\varepsilon_A$ is natural in $A$. 
		
		Surjectivity of $\varepsilon_A$ follows from existence of $\pi$-expansions in terms of coordinates due to $\pi$-adic completeness. Injectivity follows from uniqueness of coordinates when $\operatorname{Ann}_A(\pi) = \pi^m A$, which is a consequence of $A$ being strict.
	\end{proof}
	
	Finally, we verify the triangle identities.
	\begin{proposition}
		For all $\pi$-adically complete residually perfect $\mathcal{L}_{m,q}^{(t)}$-algebras $A$ and all perfect $\underline{\mathcal{L}_m}$-algebras $k$, the diagrams 
		\begin{equation*}
			\begin{tikzcd}[column sep = 10ex]
				W(k) \arrow[r, "W(\eta_k)"]\arrow[rd, "\operatorname{id}_{W(k)}"']& WUW(k)\arrow[d, "\varepsilon_{W(k)}"] \\
				& W(k) 
			\end{tikzcd}\quad 
			\begin{tikzcd}[column sep = 10ex]
				U(A) \arrow[r, "\eta_{U(A)}"]\arrow[rd, "\operatorname{id}_{U(A)}"']& UWU(A)\arrow[d, "U(\varepsilon_A)"] \\
				& U(A) 
			\end{tikzcd}
		\end{equation*}
		commute.
	\end{proposition}
	\begin{proof}
		The composition of maps in the left diagram is given by
		\begin{equation*}
			(\alpha_i)_i \mapsto (\overline{(\alpha_i,0,\dots)})_i \mapsto \sum_i [\overline{(\alpha_i,0,\dots)}^{q^{-i}}]\pi^i = \sum_i [{\overline{(\alpha_i^{q^{-i}},0,\dots)}}]\pi^i
		\end{equation*}
		noting that the Frobenius acts on coordinates in $W(k)$. This reduces to the identity on residue rings $UW(k) \to UW(k)$, and hence the left diagram commutes by faithfulness of the functor $U$. 
		
		The composition of maps in the right diagram is given by
		\begin{equation*}
			\alpha \mapsto (\alpha,0,\dots) \mapsto \overline{[\alpha] + \sum_{i\geq 1}[0^{q^{-i}}]\pi^i} = \alpha
		\end{equation*}
		since $\overline{[\alpha]} = \alpha$ and  $\pi = 0$ in the residue  ring $U(A)$.
	\end{proof}
	\begin{remark}
		If we let $\mathsf{crp}_{\mathcal{L}_{m,q}^{(t)}}$ be the category $\mathsf{CRP}_{\mathcal{L}_{m,q}^{(t)}}$ with the strictness condition removed, then we obtain a diagram
		\begin{equation*}
			\begin{tikzcd}
				\mathsf{crp}_{\mathcal{L}_{m,q}^{(t)}}\arrow[ d, shift left = 1ex, "U"] &\mathsf{CRP}_{\mathcal{L}_{m,q}^{(t)}}\arrow[l,hook']\arrow[ d, shift left = 1ex, "U"]\\
				\mathsf{PerfAlg}_{\underline{\mathcal{L}_m}}\arrow[u, "W", shift left = 1ex, "\dashv"']&\mathsf{PerfAlg}_{\underline{\mathcal{L}_m}} \arrow[l, equal]\arrow[u, "W", shift left = 1ex, "\dashv"']
			\end{tikzcd}
		\end{equation*}
		in which both columns are obtained from our newly minted adjunction; in the left column, we merely lose injectivity of $\varepsilon_A: WU(A) \to A$ in general. Thus the strictness condition can be removed while retaining the adjunction. On the other hand, both perfectness of the residue ring $U(A)$ and $\pi$-adic completeness of $A$ are necessary to define $\varepsilon_A$. Thus neither of those two properties can be removed.
		
		Since we did not find any practical use for $\mathsf{crp}_{\mathcal{L}_{m,q}^{(t)}}$, we neglected to introduce it outside of this remark. However, it does explain why it is more natural to view $W$ as a left adjoint to $U$ in the adjoint equivalence of the right column, as opposed to viewing $W$ as a right adjoint to $U$.
	\end{remark}
	\section{Lazardian jet algebras}\label{Sec:LazardianJetAlgebras}
	The goal for this section is to construct an adjunction
	\begin{equation*}
		\begin{tikzcd}
			\mathsf{Alg}_{\mathcal{L}_{m,q}^{(t)}}\arrow[from = d, shift right = 1ex, "W"',dashed] \\
			\mathsf{Alg}_{\underline{\mathcal{L}_m}}\arrow[from = u, "\Delta"', shift right= 1ex, "\dashv",dashed]
		\end{tikzcd}
	\end{equation*}
	which completes the proof of Theorem~\ref{Thm:LazardianWittVectors}.
	
	The $\pi$-typical Witt vector functor $\mathsf{Alg}_{O_K} \to \mathsf{Alg}_{O_K}$ and the power series functor $\mathsf{Alg}_{\mathbf{Z}}\to \mathsf{Alg}_{\mathbf{Z}}$ both admit left adjoints. To construct the points of the functor in the former case, we may take quotients of the free $\delta_{\pi}$-ring on a set;  in the latter, we have the Hasse-Schmidt derivation algebra as in Vojta~\cite{Voj07}, alternatively known as the jet algebra~\cite{LM09,MP24}. Both of these constructions can be understood in the generality of plethories~\cite{BW05}. Using the polynomials defining the ring operations on $W^{(t)}_{m,q} = \mathbf{A}^{m + 1}_{\underline{\mathcal{L}_m}}$, in the language of Tall and Wraith~\cite{TW70}, the ring $\Delta = \underline{\mathcal{L}_m}[X_0,\dots,X_{m }]$ can be viewed as an $ \underline{\mathcal{L}_m}$-$\mathcal{L}_{m,q}^{(t)}$-biring. We are interested in the composition products $\Delta(A) = \Delta \odot A$. The results of this section can be understood entirely within this greater generality, but we will be more concrete. 
	
	Allowing the $\mathcal{L}_{m,q}^{(t)}$-algebra $A$ to vary, the biring $\Delta$ can be viewed as a functor which we will use to construct our desired adjunction.
	
	\begin{definition}\label{Defn:LazardianJetAlgebra}
		Let $A$ be an $\mathcal{L}^{(t)}_{m,q}$-algebra. Define the $A$-valued \emph{Lazardian jet algebra} by
		\begin{equation*}
			\Delta^{(t)}_{m,q}(A) := \underline{\mathcal{L}_m}[d^{[n]}x \mid x \in A, 0\leq n \leq m ]/{\sim},
		\end{equation*}
		where $\sim$ is the ideal generated by the relations
		\begin{equation*}
			d^{[n]}(x \ast y) = Q_{n,q}^{\ast(t)}(d^{[0]}x,\dots,d^{[n]}x;d^{[0]}y,\dots,d^{[n]}y)
		\end{equation*}
		for all $x,y \in A$ and $\ast$ representing either $+$ or $\times$, and
		\begin{equation*}
			d^{[n]}x = x_n:=\varepsilon_{\mathcal{L}_{m,q}^{(t)}}^{-1}(x)_n
		\end{equation*}
		for all $x \in \mathcal{L}_{m,q}^{(t)}$, where $\varepsilon: WU \to 1_{\mathsf{CRP}_{\mathcal{L}^{(t)}_{m,q}}}$ is the counit of the adjunction $W\dashv U$ as in Proposition~\ref{Prop:WUcounit}. 
	\end{definition}
	For notational brevity we will simply write $\Delta := \Delta^{(t)}_{m,q}$.
	\begin{example}
		Consider the case $t = 0$, $m = 1$, and $q = p$. In this case we write $dx := d^{[1]}x$ and $\omega := \omega_1$ so that $\overline{p/\pi}= \omega$. Note that 
		\begin{equation*}
			\Delta(A) = \mathbf{F}_p[\omega]^{\operatorname{pf}}[dx\mid x \in A]/{\sim}
		\end{equation*}
		with $\sim$ being determined by the additive relations
		\begin{equation*}
			d(x + y) = dx + dy + \omega^p C_p(x,y),
		\end{equation*}
		the multiplicative relations
		\begin{equation*}
			d(xy) = x^p\,dy + y^p\,dx,
		\end{equation*}
		and the relation $d\pi^n = \delta_{n1}$. In particular 
		\begin{equation*}
			dp = \pi^p\,d(p/\pi) + (p/\pi)^p \,d\pi = \omega^p. 
		\end{equation*}
		Thus the additive relation becomes
		\begin{equation*}
			d(x + y) = dx + dy  + C_p(x,y)\,dp.
		\end{equation*}
		These relations are used in the construction of total $p$-differentials~\cite{DKRZ19} and Frobenius-Witt differentials~\cite{Sai22b}. A slight variation of these relations arise when working with the object $\mathbf{\Omega}_A$ as in Gabber and Ramero~\cite[9.6.12]{GR18}. The corresponding additive relation for $\mathbf{\Omega}_A$ is
		\begin{equation}\label{Equation:GRunweightedVersion}
			d(x + y) = dx + dy + C_p(x^{1/p},y^{1/p})\,dp. 
		\end{equation}
		If we did not normalize the indeterminates when defining $Q_i^+$, then we would have instead obtained
		\begin{equation*}
			Q_{1,\text{unnormalized}}^+ = X_1 + Y_1 + (p/\pi)^pC_p(X_0^{1/p},Y_0^{1/p}),
		\end{equation*}
		yielding the relation (\ref{Equation:GRunweightedVersion}) from Gabber and Ramero. 
	\end{example}

	 We show that the construction of $\Delta(A)$ is functorial.
	\begin{proposition}\label{Prop:jetAlgebraFunctoriality}
		Let $\phi: A \to A'$ be an $\mathcal{L}_{m,q}^{(t)}$-algebra morphism. Define $\Delta(\phi): \Delta(A) \to \Delta(A')$ as the unique map making the diagram
		\begin{equation*}
			\begin{tikzcd}[column sep = 15ex]
				\underline{\mathcal{L}_m}[d^{[n]}x\mid x \in A, 0 \leq n \leq m]\arrow[r, "d^{[n]}x\mapsto d^{[n]}\phi(x)"]\arrow[d] & \underline{\mathcal{L}_m}[d^{[n]}x\mid x \in A', 0 \leq n \leq m]\arrow[d] \\
				\Delta(A)\arrow[r,dashed, "\Delta(\phi)"] & \Delta(A')
			\end{tikzcd}
		\end{equation*}
		commute. This data makes $\Delta: \mathsf{Alg}_{\mathcal{L}_{m,q}^{(t)}}\to \mathsf{Alg}_{\underline{\mathcal{L}_m}}$ a functor.
	\end{proposition}
	\begin{proof}
		We have $d^{[n]}\phi(x) = d^{[n]}x = x_n$ for $x \in \mathcal{L}_{m,q}^{(t)}$ because $\phi$ is an $\mathcal{L}^{(t)}_{m,q}$-algebra map. We have
		\begin{equation*}
			d^{[n]}(\phi(x\ast y))  = Q^{\ast(t)}_{n,q}(d^{[i]}\phi(x);d^{[i]}\phi(y)) = \Delta(\phi)(Q_{n,q}^{\ast(t)}(d^{[i]}x;d^{[i]}y))
		\end{equation*}
		because $\phi(x\ast y) = \phi(x)\ast \phi(y)$ and because $Q_{n,q}^{\ast(t)}$ has coefficients in $\underline{\mathcal{L}_m}$. Functoriality follows from the universality of quotients.
	\end{proof}
	
	We construct the counit of the adjunction.
	\begin{proposition}
		Let $k$ be an $\underline{\mathcal{L}_m}$-algebra. Define $\varepsilon'_k: \Delta W(k) \to k$ as the unique map making the diagram
		\begin{equation*}
			\begin{tikzcd}
				\underline{\mathcal{L}_m}[d^{[n]}x \mid x \in W(k), 0 \leq n \leq m] \arrow[rd, "d^{[n]}(x_i)_i \mapsto x_n"]\arrow[d]& \\
				\Delta W(k) \arrow[r,dashed, "\varepsilon'_k"]& k
			\end{tikzcd}
		\end{equation*}
		commute. These maps define a natural transformation 
		\begin{equation*}
			\varepsilon': \Delta W \to 1_{\mathsf{Alg}_{\underline{\mathcal{L}_m}}}. 
		\end{equation*}
	\end{proposition}
	\begin{proof}
		We have $\varepsilon'_k(d^{[n]}x) = x_n = \varepsilon'_k(x_n) $ for all $x \in \mathcal{L}_{m,q}^{(t)}$ since $k$ is an $\underline{\mathcal{L}_m}$-algebra. We have
		\begin{equation*}
			\varepsilon'_k(d^{[n]}(x\ast y)) = Q^{\ast(t)}_{n,q}(x_i;y_i) = \varepsilon'_kQ_{n,q}^{\ast(t)}(d^{[i]}x;d^{[i]}y)
		\end{equation*}
		since $Q_{n,q}^{\ast(t)}$ has coefficients in $\underline{\mathcal{L}_m}$. Thus we obtain our factorization $\Delta W(k) \to k$. Naturality follows from the fact that $W(\phi)$ applies $\phi$ to each coordinate. 
	\end{proof}
	We construct the unit of the adjunction.
	\begin{proposition}
		Let $A$ be an $\mathcal{L}_{m,q}^{(t)}$-algebra. The function $\eta'_A: A \to W\Delta(A)$ defined by $a \mapsto (d^{[n]}a)_n$ is an $\mathcal{L}_{m,q}^{(t)}$-algebra morphism natural in $A$, thus yielding a natural transformation \begin{equation*}
			\eta': 1_{\mathsf{Alg}_{\mathcal{L}_{m,q}^{(t)}}} \to W\Delta. 
		\end{equation*}
	\end{proposition}
	\begin{proof}
		We calculate $x \mapsto (d^{[n]}x)_n = (x_n)_n = x$ for all $x \in \mathcal{L}_{m,q}^{(t)}$. We calculate
		\begin{equation*}
			a \ast b \mapsto (d^{[n]}(a\ast b))_n = (Q_{n,q}^{\ast (t)}(d^{[i]}a;d^{[i]}b))_n = (d^{[n]}a)_n\ast (d^{[n]}b)_n.
		\end{equation*}
		Thus $\eta'_A(a\ast b) = \eta'_A(a)\ast \eta'_A(b)$. Naturality follows from $d^{[n]}\phi(a) = \Delta(\phi)(d^{[n]}a)$ and the fact that $W\Delta(\phi)$ is applied on each coordinate.
	\end{proof}
	Finally, we verify the triangle identities.
	\begin{proposition}
		For all $\mathcal{L}_{m,q}^{(t)}$-algebras $A$ and $\underline{\mathcal{L}_m}$-algebras $k$, the diagrams 
		\begin{equation*}
			\begin{tikzcd}[column sep = 10ex]
				\Delta(A) \arrow[r, "\Delta(\eta'_A)"]\arrow[rd, "\operatorname{id}_{\Delta(A)}"']& \Delta W\Delta (A)\arrow[d, "\varepsilon'_{\Delta(A)}"] \\
				& \Delta(A) 
			\end{tikzcd}\quad 
			\begin{tikzcd}[column sep = 10ex]
				W(k) \arrow[r, "\eta'_{W(k)}"]\arrow[rd, "\operatorname{id}_{W(k)}"']& W\Delta W(k)\arrow[d, "W(\varepsilon'_k)"] \\
				& W(k) 
			\end{tikzcd}
		\end{equation*}
		commute.
	\end{proposition}
	\begin{proof}
		The composition of maps in the left diagram is determined by
		\begin{equation*}
			d^{[n]}x \mapsto d^{[n]}(d^{[i]}x)_i \mapsto d^{[n]}x
		\end{equation*}
		and is thus the identity on $\Delta(A)$. 
		
		The composition of maps in the right diagram is given by
		\begin{equation*}
			x \mapsto (d^{[n]}x)_n \mapsto (x_n)_n = x
		\end{equation*}
		and is thus the identity on $W(k)$.
	\end{proof}
	This completes the proof of Theorem~\ref{Thm:LazardianWittVectors}.
	\section{Universal residual perfection for Lazardian $O$-algebras}\label{Sec:ThreeFunctors1}
	We prove Theorem~\ref{Thm:LazardianURP}. Let $O$ be an object of $\mathsf{CRP}_{\mathcal{L}^{(t)}_{m,q}}$, and let $A$ be a Lazardian $O$-algebra, which we recall is a strict $\pi$-adically complete $O$-algebra. We assume the data of the diagram
	\begin{equation*}
		\begin{tikzcd}
			\mathsf{Alg}_{O}\arrow[from = d, shift right = 1ex, "W"'] &\mathsf{CRP}_{O}\arrow[l,hook']\arrow[ d, shift left = 1ex, "U"]\\
			\mathsf{Alg}_{\underline{O}}\arrow[from = u, "\Delta"', shift right= 1ex, "\dashv"]&\mathsf{PerfAlg}_{\underline{O}}\arrow[l,hook']\arrow[u, "W", shift left = 1ex, "\dashv"']
		\end{tikzcd}
	\end{equation*}
	and denote the two adjunctions by $(\Delta,W,\eta',\varepsilon')$ and $(W,U,\eta,\varepsilon)$. We also assume that $\eta$ and $\varepsilon$ are natural isomorphisms. 
	
	We first show that restriction of scalars $\mathsf{Alg}_A \to \mathsf{Alg}_O$ induces a factorization $\mathsf{CRP}_A \to \mathsf{CRP}_O$. 
	\begin{proposition}
		Let $A$ be a Lazardian $O$-algebra. An $A$-algebra $B$ is an object of $\mathsf{CRP}_A$ if and only if it is an object of $\mathsf{CRP}_{O}$. In other words, restriction of scalars $\mathsf{Alg}_A \to \mathsf{Alg}_O$ fits into the diagram
		\begin{equation*}
			\begin{tikzcd}
				&&\mathsf{Alg}_A\arrow[lld]\\
				\mathsf{Alg}_{O} &\mathsf{CRP}_{O}\arrow[l,hook']& \mathsf{CRP}_A.\arrow[l,dashed]\arrow[u,hook]
			\end{tikzcd}
		\end{equation*}
	\end{proposition}
	\begin{proof}
		Both $\pi$-adic completeness and residual perfectness are independent of the base. It remains to verify the strictness condition. Since $A$ is strict, we have an isomorphism $\pi O\otimes_{O} A \xrightarrow{\sim}\pi A$, and base change gives an isomorphism 
		\begin{equation*}
			\pi O\otimes_{O} B \simeq\pi O\otimes_{O} A\otimes_A B \xrightarrow{\sim}\pi A\otimes_A B.
		\end{equation*}
		Since we have a commuting diagram
		\begin{equation*}
			\begin{tikzcd}
				\pi O\otimes_{O}B\arrow[r]\arrow[d,"\sim"'] & B, \\
				\pi A\otimes_A B\arrow[ru]
			\end{tikzcd}
		\end{equation*}
		the top arrow is injective if and only if the bottom arrow is injective.
	\end{proof}
	
	We now define the factorization 
	\begin{equation*}
		\begin{tikzcd}
			\mathsf{CRP}_{O}\arrow[ d, shift left = 1ex, "U"]&\mathsf{CRP}_A \arrow[d, shift left = 1ex, "U", dashed] \arrow[l]\\
			\mathsf{PerfAlg}_{\underline{O}} \arrow[u, "W", shift left = 1ex, "\dashv"']&\mathsf{PerfAlg}_{\Delta(A)^{\operatorname{pf}}}\arrow[u, "W", shift left = 1ex, "\dashv"',dashed]\arrow[l]
		\end{tikzcd}
	\end{equation*}
	of $W$ and $U$ on objects.
	\begin{definition}\label{Definition:Delta(A)pfStructure}
		Let $A$ be a Lazardian $O$-algebra. For each object $B$ of $\mathsf{CRP}_A$, equip the object $U(B)$ of $\mathsf{PerfAlg}_{\underline{O}}$ with the unique $\Delta(A)^{\operatorname{pf}}$-algebra structure making the diagram
		\begin{equation*}
			\begin{tikzcd}
				\Delta W(\underline{O})\arrow[rrrr]\arrow[rd,"\varepsilon'_{\underline{O}}"]\arrow[dddd, "\Delta(\varepsilon_{O})"']& & && \Delta W U (B)\arrow[ld, "\varepsilon'_{U(B)}"']\arrow[dddd,"\sim"', "\Delta(\varepsilon_B)"]\\
				 & \underline{O}\arrow[rr] & &U(B)\\
				 &&\Delta(A)^{\operatorname{pf}}\arrow[ru,dashed, "\exists !"]\\
				 &&\Delta(A)\arrow[u]\arrow[rrd]\\
				 \Delta(O) \arrow[rrrr]\arrow[rru]&&&& \Delta(B)
			\end{tikzcd}
		\end{equation*}
		commute. For each object $k$ of $\mathsf{PerfAlg}_{\Delta(A)^{\operatorname{pf}}}$, equip $W(k)$ with the $A$-algebra structure $A \to W\Delta(A) \to  W(\Delta(A)^{\operatorname{pf}})\to W(k)$.
	\end{definition}
	We only needed the right trapezium of the diagram in Definition~\ref{Definition:Delta(A)pfStructure} to define $\Delta(A)^{\operatorname{pf}}\to U(B)$, but we included the whole diagram to provide some context.
	
	Having defined the factorization of $W$ and $U$ on objects, we now require $W$ and $U$ to be compatible with morphisms.
	
	\begin{proposition}
		The functor $W: \mathsf{PerfAlg}_{\Delta(A)^{\operatorname{pf}}} \to \mathsf{CRP}_{O}$ factors through $\mathsf{CRP}_A$.
	\end{proposition}
	\begin{proof}
		If $k \to l$ is a morphism of $\mathsf{PerfAlg}_{\Delta(A)^{\operatorname{pf}}}$, then $W(k) \to W(l)$ is a $W(\Delta(A)^{\operatorname{pf}})$-algebra morphism and is thus an $A$-algebra morphism via $A \to W(\Delta(A)^{\operatorname{pf}})$.
	\end{proof}
	
	\begin{proposition}
		The functor $U: \mathsf{CRP}_A \to \mathsf{PerfAlg}_{\underline{O}}$ factors through $\mathsf{PerfAlg}_{\Delta(A)^{\operatorname{pf}}}$. 
	\end{proposition}
	\begin{proof}
		Let $B \to C$ be a morphism of $\mathsf{CRP}_A$. Consider the diagram
		\begin{equation*}
			\begin{tikzcd}
				\Delta WU(B)\arrow[rrrr]\arrow[rd,"\varepsilon'_{U(B)}"]\arrow[dddd,"\sim", "\Delta(\varepsilon_B)"']& & && \Delta W U (C)\arrow[ld, "\varepsilon'_{U(C)}"']\arrow[dddd,"\sim"', "\Delta(\varepsilon_C)"]\\
				& U(B)\arrow[rr] & &U(C)\\
				&&\Delta(A)^{\operatorname{pf}}\arrow[ru]\arrow[lu]\\
				&&\Delta(A)\arrow[u]\arrow[rrd]\arrow[lld]\\
				\Delta(B) \arrow[rrrr]&&&& \Delta(C)
			\end{tikzcd}
		\end{equation*}
		in which only commutativity of the upper triangle is in question. Commutativity of this triangle follows from $\Delta(A)^{\operatorname{pf}}$ being initial in $\mathsf{PerfAlg}_{\Delta(A)}$. 
	\end{proof}
	Thus we have constructed factorizations $W: \mathsf{PerfAlg}_{\Delta(A)^{\operatorname{pf}}} \to \mathsf{CRP}_{A}$ and $U:  \mathsf{CRP}_A \to \mathsf{PerfAlg}_{\Delta(A)^{\operatorname{pf}}}$. It remains to show that the unit and counit isomorphisms preserve the relevant algebra structures, which establishes the required adjoint equivalence.
	\begin{proposition}
		For each perfect $\Delta(A)^{\operatorname{pf}}$-algebra $k$, the unit $\eta_k: k \to UW(k)$ is a $\Delta(A)^{\operatorname{pf}}$-algebra morphism. 
	\end{proposition}
	\begin{proof}
		We have a commutative diagram
		\begin{equation*}
			\begin{tikzcd}[column sep = 7.5ex]
				\Delta(A)\arrow[rd, equal]\arrow[r, "\Delta(\eta'_A)"] & \Delta W\Delta (A)\arrow[r]\arrow[d, "\varepsilon'_{\Delta(A)}"] & \Delta W (\Delta(A)^{\operatorname{pf}})\arrow[r]\arrow[d, "\varepsilon'_{\Delta(A)^{\operatorname{pf}}}"] & \Delta W(k)\arrow[d, "\varepsilon'_k"] \\
				& \Delta(A)\arrow[r] & \Delta(A)^{\operatorname{pf}}\arrow[r] & k, 
			\end{tikzcd}
		\end{equation*}
		where $\Delta(A)^{\operatorname{pf}} \to k$ is the given structure map. The top row is by definition $\Delta$ applied to the structure map $A \to W(k)$. This diagram becomes the left square in the diagram
		\begin{equation*}
			\begin{tikzcd}[column sep = 10ex]
				\Delta(A) \arrow[r]\arrow[d] & \Delta W(k) \arrow[r, "\Delta W(\eta_k)"', shift right = 0.5ex]\arrow[d, "\varepsilon_k'"]& \Delta WUW(k) \arrow[d, "\varepsilon'_{UW(k)}"]\arrow[l, shift right = 0.5ex, "\Delta(\varepsilon_{W(k)})"']\\
				\Delta(A)^{\operatorname{pf}} \arrow[r]\arrow[rr, bend right = 5ex]& k \arrow[r, "\eta_k", "\sim"']& UW(k)
			\end{tikzcd}
		\end{equation*}
		in which only commutativity of the bottom semicircle is in question. Note that the outer rounded rectangle commutes by definition of the $\Delta(A)^{\operatorname{pf}}$-algebra structure on $UW(k)$. Since $k\xrightarrow{\sim}UW(k)$ is perfect via the isomorphism $\eta_k$, and since the two maps $\Delta(A)^{\operatorname{pf}} \to UW(k)$ forming the bottom semicircle are equalized on $\Delta(A)$, the bottom semicircle commutes because $\Delta(A)^{\operatorname{pf}}$ is initial in $\mathsf{PerfAlg}_{\Delta(A)}$. Thus $\eta_k$ is a $\Delta(A)^{\operatorname{pf}}$-algebra morphism.
	\end{proof}
	\begin{proposition}
		For each object $B$ of $\mathsf{CRP}_A$, the counit $\varepsilon_B: WU(B) \to B$ is an $A$-algebra morphism. 
	\end{proposition}
	\begin{proof}
		Consider the diagram
		\begin{equation*}
			\begin{tikzcd}
				WU(B)\arrow[rrr, "\eta'_{WU(B)}", shift left = 0.5ex] \arrow[ddd, "\varepsilon_B"', "\sim"]& & & W \Delta W U(B)\arrow[ddd, "W\Delta(\varepsilon_B)", "\sim"']\arrow[lll, shift left = 0.5ex, "W(\varepsilon'_{U(B)})"]\\
				&  W(\Delta(A)^{\operatorname{pf}})\arrow[lu]\\
				& A\arrow[ld]\arrow[r, "\eta'_A"] & W\Delta(A)\arrow[rd]\arrow[lu]\\
				B \arrow[rrr, "\eta'_B"]& & & W\Delta (B)
			\end{tikzcd}
		\end{equation*}
		in which only commutativity of the left crooked quadrilateral is in question. Note that the right upper triangle commutes because it is $W$ applied to the structure map $\Delta(A) \to \Delta(A)^{\operatorname{pf}} \to U(B)$. Utilizing the facts that the outer rectangle commutes and that $\varepsilon_B$ is an isomorphism, we deduce that the left crooked quadrilateral commutes. Thus $\varepsilon_B$ is an $A$-algebra morphism. 
	\end{proof}
	\begin{theorem}
		Let $O$ be an object of $\mathsf{CRP}_{\mathcal{L}_{m,q}^{(t)}}$ and let $A$ be  a Lazardian $O$-algebra. The functors $W: \mathsf{PerfAlg}_{\Delta(A)^{\operatorname{pf}}} \to \mathsf{CRP}_A$ and $U: \mathsf{CRP}_A \to \mathsf{PerfAlg}_{\Delta(A)^{\operatorname{pf}}}$  along with the unit $\eta: 1_{\mathsf{PerfAlg}_{\Delta(A)^{\operatorname{pf}}}} \to UW$ and counit $\varepsilon: WU \to 1_{\mathsf{CRP}_A}$ form an adjoint equivalence of categories.
	\end{theorem}
	\begin{proof}
		The unit and counit are natural isomorphisms of $\Delta(A)^{\operatorname{pf}}$-algebras and $A$-algebras, respectively.
	\end{proof}
	This completes the proof of Theorem~\ref{Thm:LazardianURP}.

	\section{An explicit calculation in equal characteristic}
	
	\begin{definition}
		Let $A$ be a ring. Define the $A$-valued $m$th order Hasse-Schmidt derivation algebra by
		\begin{equation*}
			\operatorname{HS}^m(A) := \mathbf{Z}[d^{[n]}a\mid a \in A, 0\leq n \leq m]/{\sim},
		\end{equation*}
		where $\sim$ is the equivalence relation generated by the relations
		\begin{equation*}
			d^{[n]}(a + b) = d^{[n]}a + d^{[n]}b,\quad d^{[n]}(ab) = \sum_{i + j = n}d^{[i]}a\,d^{[j]}b,\quad d^{[0]}1 = 1.
		\end{equation*}
	\end{definition}
	
	Set $O_m := \mathbf{F}_p\llbracket\pi\rrbracket/(\pi^{m + 1})$ for some $m \in \mathbf{N}\cup \{\infty\}$, and  set $U:= -\otimes_{O_m}\mathbf{F}_p$. We will take the diagram
	\begin{equation*}
		\begin{tikzcd}[column sep = 25ex]
			\mathsf{Alg}_{O_m}\arrow[from = d, shift right = 1ex, "(-)\llbracket\pi\rrbracket/(\pi^{m + 1})"'] &\mathsf{CRP}_{O_m}\arrow[l,hook']\arrow[ d, shift left = 1ex, "U"]\\
			\mathsf{Alg}_{\mathbf{F}_p}\arrow[from = u, "\operatorname{HS}^m(-)\otimes_{\operatorname{HS}^m(O_m)}\mathbf{F}_p"', shift right= 1ex, "\dashv"]&\mathsf{PerfAlg}_{\mathbf{F}_p} \arrow[l,hook']\arrow[u, "(-)\llbracket\pi\rrbracket/(\pi^{m + 1})", shift left = 1ex, "\dashv"']
		\end{tikzcd}
	\end{equation*}
	along with the fact that the right column forms an adjoint equivalence for granted so that Theorem~\ref{Thm:LazardianURP} gives
	\begin{equation*}
		k^u = (\operatorname{HS}^m(A)\otimes_{\operatorname{HS}^m(O_m)}\mathbf{F}_p)^{\operatorname{pf}}.
	\end{equation*} 
	We have a simplification if $A = k\llbracket\pi\rrbracket/(\pi^{m + 1})$ for some $\mathbf{F}_p$-algebra $k$. In this case we treat elements $a = \sum_{i = 0}^m a_i\pi^i$ as a vector of coordinates $a_i \in k$ without mention. For the $m < \infty$ case, the notation $a_i$ for $i > m$ may inadvertently arise. In this case we take $a_i := 0$. 
	\begin{proposition}\label{Prop:HS(k)Surjects}
		Let $A:= k\llbracket\pi\rrbracket/(\pi^{m + 1})$. We have a commutative diagram
		\begin{equation*}
			\begin{tikzcd}
				\operatorname{HS}^m(k) \arrow[r,hook]\arrow[rd, two heads]& \operatorname{HS}^m(A) \arrow[r]\arrow[d]& \operatorname{HS}^m(A)\otimes_{\operatorname{HS}^m(O_m)}\mathbf{F}_p\arrow[ld, shift left = 0.5ex] \\
				&\operatorname{HS}^m(A)/(d^{[n]}\pi^i = \delta_{ni})\arrow[ru, shift left = 0.5ex]
			\end{tikzcd}
		\end{equation*}
		in which the two bottom right arrows are mutually inverse. 
	\end{proposition}
	\begin{proof}
		We show that the bottom left arrow is a surjection. Given $a \in A$, there exists $z \in A$ such that
		\begin{equation*}
			d^{[n]} a = d^{[n]} \left(\sum_{i = 0}^n a_i\pi^i + \pi^{n + 1}z\right) = \sum_{i = 0}^n d^{[n]}(a_i\pi^i) + d^{[n]}(\pi^{n+ 1}z).
		\end{equation*}
		The Leibniz rule and $d^{[\alpha]}\pi^{n + 1} = \delta_{\alpha,n + 1} = 0$ for  $\alpha < n + 1$ imply that
		\begin{equation*}
			d^{[n]}(\pi^{n + 1}z) = \sum_{\alpha+\beta = n}d^{[\alpha]}\pi^{n + 1}d^{[\beta]}z = 0.
		\end{equation*}
		On the other hand
		\begin{equation*}
			d^{[n]}(a_i\pi^i) = \sum_{\alpha + \beta = n} d^{[\alpha]}a_id^{[\beta]}\pi^i = d^{[n-i]}a_{i}. 
		\end{equation*}
		Therefore
		\begin{equation}
			d^{[n]}\sum_{i = 0}^m a_i\pi^i = \sum_{i = 0}^n d^{[n-i]}a_i,\label{Eqn:SurjectionFormula}
		\end{equation}
		establishing surjectivity. If $a \in O_m$, then Equation~(\ref{Eqn:SurjectionFormula}) implies that $d^{[n]}a \mapsto a_n$ because $d^{[\alpha]}n = 0$ for all $n \in \mathbf{Z}$ when  $\alpha > 0$. Thus universality of the coproduct gives the map $\operatorname{HS}^m(A)\otimes_{\operatorname{HS}^m(O_m)}\mathbf{F}_p \to \operatorname{HS}^m(A)/(d^{[n]}\pi^i = \delta_{ni})$. We have a map in the other direction because $d^{[n]}\pi^i\otimes 1 = 1\otimes \delta_{ni} = \delta_{ni}$. By inspection, these two maps are mutually inverse.
	\end{proof}
	Consider the diagram
	\begin{equation*}
		\begin{tikzcd}
			\operatorname{HS}^m(k) \arrow[r,hook]\arrow[rd, two heads]& \operatorname{HS}^m(A) \arrow[d] \\
			&\operatorname{HS}^m(A)/(d^{[n]}\pi^i = \delta_{ni})
		\end{tikzcd}
	\end{equation*}
	from Proposition~\ref{Prop:HS(k)Surjects}. We will construct a retraction of the surjection in this diagram, which is equivalent to constructing a retraction of $\operatorname{HS}^m(k) \hookrightarrow \operatorname{HS}^m(A)$ factoring through the quotient. The formula (\ref{Eqn:SurjectionFormula}) that we used to prove surjectivity is a promising candidate for our desired retraction, so long as we can show that it is a homomorphism. We first prove a universal identity.
	\begin{lemma}\label{Lemma:HSUniversalIdentity}
		In the polynomial ring $\mathbf{Z}[x_{ij},y_{ij} \mid i,j \in \mathbf{N}]$, we have
		\begin{equation*}
			\sum_{\gamma = 0}^n\sum_{\varepsilon = 0}^\gamma\sum_{\alpha = 0}^{n-\gamma}x_{\alpha \varepsilon} y_{n-\gamma-\alpha,\gamma-\varepsilon}= \sum_{t = 0}^n \sum_{i = 0}^t \sum_{j = 0}^{n-t} x_{t-i,i}y_{n-t-j,j}
		\end{equation*}
		for all $n \in \mathbf{N}$.
	\end{lemma}
	\begin{proof}
		Start with the left-hand side. Relabeling $i:= \varepsilon$ and $t:= \alpha$, we have
		\begin{equation*}
			\sum_{\gamma = 0}^n\sum_{\varepsilon = 0}^\gamma \sum_{\alpha = 0}^{n-\gamma} x_{\alpha\varepsilon}y_{n-\gamma-\alpha,\gamma-\varepsilon} =\sum_{\gamma = 0}^n \sum_{i = 0}^\gamma \sum_{t= 0}^{n-\gamma} x_{t,i}y_{n-t-\gamma,\gamma-i}.
		\end{equation*}
		Increasing the upper and lower bounds of the $t$-indexed sum by $i$ gives
		\begin{equation*}
			\sum_{\gamma = 0}^n\sum_{\varepsilon = 0}^\gamma \sum_{\alpha = 0}^{n-\gamma} x_{\alpha\varepsilon}y_{n-\gamma-\alpha,\gamma-\varepsilon} =\sum_{\gamma = 0}^n\sum_{i = 0}^\gamma \sum_{t= i}^{n-\gamma + i} x_{t-i,i}y_{n-t-\gamma + i,\gamma-i}.
		\end{equation*}
		We now wish to change the order of summation. We have the conditions
		\begin{equation*}
			0 \leq i \leq \gamma \leq n, \quad i \leq t \leq n - (\gamma-i)\leq n, \quad n-t - \gamma + i \geq 0.
		\end{equation*}
		Since $i = 0$ occurs, we have $0 \leq t \leq n$. We now determine how $i$ depends on $t$. Since $\gamma = 0$ occurs and $t-n \leq 0$, we have $0 \leq i \leq t$. It remains to check how $\gamma$ depends on $i$ and $t$. We have
		\begin{equation*}
			i \leq \gamma \leq n,\quad i \leq \gamma, \quad \gamma \leq i + n-t\leq n
		\end{equation*}
		since $i-t \leq 0$. Hence $i \leq \gamma \leq i + n-t$. Thus
		\begin{equation*}
			\sum_{\gamma = 0}^n\sum_{\varepsilon = 0}^\gamma \sum_{\alpha = 0}^{n-\gamma} x_{\alpha\varepsilon}y_{n-\gamma-\alpha,\gamma-\varepsilon} = \sum_{t = 0}^n \sum_{i = 0}^t \sum_{\gamma = i}^{i + n -t} x_{t-i,i}y_{n-t-\gamma + i,\gamma-i}. 
		\end{equation*}
		Decreasing the upper and lower bounds of the $\gamma$-index sum by $i$ gives 
		\begin{equation*}
			\sum_{\gamma = 0}^n\sum_{\varepsilon = 0}^\gamma \sum_{\alpha = 0}^{n-\gamma} x_{\alpha\varepsilon}y_{n-\gamma-\alpha,\gamma-\varepsilon} = \sum_{t = 0}^n \sum_{i = 0}^t \sum_{\gamma = 0}^{ n -t} x_{t-i,i}y_{n-t-\gamma ,\gamma},
		\end{equation*}
		and finally relabeling $j := \gamma$ gives
		\begin{equation*}
			\sum_{\gamma = 0}^n\sum_{\varepsilon = 0}^\gamma \sum_{\alpha = 0}^{n-\gamma} x_{\alpha\varepsilon}y_{n-\gamma-\alpha,\gamma-\varepsilon} =\sum_{t = 0}^n \sum_{i = 0}^t \sum_{j= 0}^{ n -t} x_{t-i,i}y_{n-t-j ,j},
		\end{equation*}
		which was to be shown.
	\end{proof}
	\begin{proposition}
		Let $A:=k\llbracket\pi\rrbracket/(\pi^{m + 1})$. The function $\phi: A \to \operatorname{HS}^m(k)\llbracket\pi\rrbracket/(\pi^{m + 1})$ defined by
		\begin{equation*}
			\phi(a):= \sum_{n = 0}^m \sum_{i = 0}^n d^{[n-i]}a_i\pi^n
		\end{equation*}
		is a ring homomorphism.
	\end{proposition}
	\begin{proof}
		We first observe that
		\begin{equation*}
			\phi(1) = \sum_{n=0}^m \sum_{i = 0}^n d^{[n-i]}\delta_{0i} \pi^n = \sum_{n = 0}^m (d^{[n]}1) \pi^n = 1,
		\end{equation*}
		noting that $d^{[n]}1 = 0$ for all $n > 0$. Additivity follows from
		\begin{equation*}
			\sum_{n = 0}^m \sum_{i = 0}^n d^{[n-i]}(a_i + b_i) \pi^n =\sum_{n = 0}^m \sum_{i = 0}^n d^{[n-i]}a_i \pi^n+\sum_{n = 0}^m \sum_{i = 0}^n d^{[n-i]}b_i \pi^n.
		\end{equation*}
		Multiplicativity is harder. On one hand, we calculate $\phi(ab)$ as
		\begin{equation*}
			 \sum_{n = 0}^m \sum_{\gamma = 0}^n d^{[n-\gamma]}\left(\sum_{\varepsilon +s = \gamma}a_\varepsilon b_s\right) \pi^n = \sum_{n = 0}^m \sum_{i = 0}^n\sum_{\varepsilon +s = \gamma}\sum_{\alpha + \beta = n-\gamma}d^{[\alpha]}a_\varepsilon \,d^{[\beta]}b_s \pi^n
		\end{equation*}
		Rewriting each sum in terms of a single increasing variable gives
		\begin{equation*}
			\phi(ab) = \sum_{n = 0}^m \sum_{\gamma = 0}^n \sum_{\varepsilon = 0}^\gamma \sum_{\alpha = 0}^{n-\gamma}d^{[\alpha]}a_{\varepsilon}\,d^{[n-\gamma - \alpha]}b_{\gamma-\varepsilon}\pi^n.
		\end{equation*}
		On the other hand, we calculate $\phi(a)\phi(b)$ as 
		\begin{equation*}
			 \sum_{n = 0}^m \sum_{i = 0}^n d^{[n-i]}a_i \pi^n\sum_{n = 0}^m \sum_{i = 0}^n d^{[n-i]}b_i \pi^n = \sum_{n = 0}^m \sum_{t + N  = n}\sum_{i = 0}^t\sum_{j = 0}^N d^{[t-i]}a_i\,d^{[N-j]}b_j \pi^n.
		\end{equation*}
		Rewriting each sum in terms of a single increasing variable gives
		\begin{equation*}
			\phi(a)\phi(b) = \sum_{n = 0}^m \sum_{t = 0}^n\sum_{i = 0}^t \sum_{j = 0}^{n-t} d^{[t-i]}a_i\,d^{[n-t-j]}b_j\pi^n.
		\end{equation*}
		But now $\phi(ab) = \phi(a)\phi(b)$ follows from the universal identity in Lemma~\ref{Lemma:HSUniversalIdentity} by taking $x_{ij}\mapsto d^{[i]}a_j$ and $y_{ij} \mapsto d^{[i]}b_j$. 
	\end{proof}
	\begin{proposition}
		Let $A:= k\llbracket\pi\rrbracket/(\pi^{m + 1})$. The transpose $\overline{\phi}: \operatorname{HS}^m(A)\to\operatorname{HS}^m(k)$ of the  ring homomorphism $\phi: A \to \operatorname{HS}^m(k)\llbracket\pi\rrbracket/(\pi^{m + 1})$ defined by
		\begin{equation*}
			\phi(a):= \sum_{n = 0}^m \sum_{i = 0}^n d^{[n-i]}a_i\pi^n
		\end{equation*}
		fits into the commutative diagram
		\begin{equation*}
			\begin{tikzcd}
				\operatorname{HS}^m(k) \arrow[r,hook, shift right = 0.5ex]\arrow[rd, two heads, shift left = 0.5ex, "\rho"]& \operatorname{HS}^m(A)\arrow[l, shift right = 0.5ex,"\overline{\phi}"'] \arrow[d] \\
				&\operatorname{HS}^m(A)/(d^{[n]}\pi^i = \delta_{ni})\arrow[lu, shift left = 0.5ex, "\overline{\phi}^\#"]
			\end{tikzcd}
		\end{equation*}
		in which $\overline{\phi}^\#$ is a retraction of $\rho$. Hence $\rho$ is injective and is thus an isomorphism.
	\end{proposition}
	\begin{proof}
		For $a \in A$, we have
		\begin{equation*}
			\overline{\phi}(d^{[n]}a)= \sum_{i = 0}^n d^{[n-i]} a_i. 
		\end{equation*}
		On the other hand an element $a \in k$ has coordinates $a_i = a\delta_{i0}$ when viewed as an element of $A$. Thus
		\begin{equation*}
			\overline{\phi}(d^{[n]}a) = \sum_{i = 0}^n d^{[n-i]}a_i = d^{[n]}a_0 = d^{[n]}a,
		\end{equation*}
		proving that $\overline{\phi}$ is a retraction to $\operatorname{HS}^m(k) \to \operatorname{HS}^m(A)$. 
		
		For the factorization through the quotient, we calculate
		\begin{equation*}
			d^{[n]}\pi^t = d^{[n]}\sum_{i = 0}^m \delta_{it} \pi^t \mapsto  \sum_{i = 0}^n d^{[n-i]}\delta_{it}. 
		\end{equation*}
		If $n < t$, then $\delta_{it} = 0$ for all $i$ and $d^{[n]}\pi^t \mapsto 0 = \delta_{nt}$. Else $ n \geq t$, which is to say that $n -t\geq 0$, and
		\begin{equation*}
			d^{[n]}\pi^t \mapsto d^{[n-t]}1 = \delta_{nt}. 
		\end{equation*}
		So $\overline{\phi}(d^{[n]}\pi^t) = \delta_{nt}$, inducing the factorization $\overline{\phi}^\#$ through the quotient. 
	\end{proof}
	\begin{corollary}\label{Cor:ExplicitPowerSeriesCalc}
		The universal residual perfection of $A := k\llbracket\pi\rrbracket/(\pi^{m + 1})$ is 
		\begin{equation*}
			A^u = \operatorname{HS}^m(k)^{\operatorname{pf}}\llbracket\pi\rrbracket/(\pi^{m + 1})
		\end{equation*} 
		with structure map
		\begin{equation*}
			a  = \sum_{n = 0}^m a_i \pi^i \mapsto \sum_{n = 0}^m \left(\sum_{i = 0}^n d^{[n-i]}a_i \right)\pi^n, 
		\end{equation*}
		inducing the inclusion $k \to \operatorname{HS}^m(k)^{\operatorname{pf}}: \overline{a}\mapsto d^{[0]}\overline{a}$ on residue rings.
	\end{corollary}
	\begin{proof}
		We have $ \operatorname{HS}^m(k)\xrightarrow{\sim}\operatorname{HS}^m(A)\otimes_{\operatorname{HS}^m(O_m)}\mathbf{F}_p $.
	\end{proof}
	Since the Cohen structure theorem implies that every positive characteristic complete discrete valuation ring $A$ with residue field $k$ is noncanonically of the form $ k\llbracket\pi\rrbracket\xrightarrow{\sim}A$, Borger's calculation that $k^u = k[u_{t,n}\mid t \in T,n\geq 1]^{\operatorname{pf}}$ for any choice of lifted $p$-basis $T \subset A$ of $k$ comes down to the fact that the elements $d^{[n]}\overline{t} \in \operatorname{HS}(k)$ over all $t \in T$ and $n \geq 1$ form an algebraically independent $k$-algebra generating set for $\operatorname{HS}(k)$.

	\bibliography{references}
	\bibliographystyle{alpha}

\end{document}